\documentclass[11pt]{amsart}
\usepackage[left=3.3cm,top=2.7cm,right=3.3cm]{geometry}
\newcommand*{\diam}{\mathord{\diamond}}
\usepackage{amscd,amssymb,graphicx,subfigure}
\usepackage{enumerate}
\usepackage{pgf,tikz}
\usetikzlibrary{arrows,shapes}
\usetikzlibrary{decorations.pathreplacing}
\usepackage{tkz-berge}
\usepackage{epstopdf}
\usepackage{color}
\usepackage[colorlinks=true, linkcolor=blue, anchorcolor=blue, citecolor=blue, filecolor=blue, menucolor= blue, urlcolor=blue]{hyperref}
\usepackage{verbatim} 
\usepackage{soul}

\usepackage[pagewise]{lineno}

\definecolor{patriarch}{rgb}{0.5, 0.0, 0.5}
\newtheorem{theorem}{Theorem}[section]

\newtheorem{corollary}[theorem]{Corollary}
\newtheorem{lemma}[theorem]{Lemma}
\newtheorem{question}[theorem]{Question}

\newtheorem{prop}[theorem]{Proposition}

\newtheorem*{theoremsp2}{Theorem \ref{thm:kdiamond}}
\newtheorem*{theoremsp3}{Theorem \ref{cyclicity}}

\newtheorem*{theoremsp5}{Theorem \ref{thm:allinthefamily}}
\theoremstyle{definition}

\newtheorem{definition}[theorem]{Definition}
\newtheorem{example}[theorem]{Example}

\title{Universal Partial Words over Non-Binary Alphabets}%
\author[Goeckner, Groothuis, Hettle, Kell, Kirkpatrick, Kirsch, and Solava]{Bennet Goeckner*, Corbin Groothuis\ddag, Cyrus Hettle\dag, Brian Kell*, \\ Pamela Kirkpatrick\dag, Rachel Kirsch\ddag, and Ryan Solava\dag}
\thanks{*Research supported in part by National Science Foundation grant \# 1604733 ``Collaborative Research: Rocky Mountain-Great Plains Graduate Research Workshop in Combinatorics''.}
\thanks{\dag Research supported in part by National Security Agency Grant \# H98230-16-1-0018 ``The 2016 Rocky Mountain - Great Plains Graduate Research Workshop in Combinatorics''.}
\thanks{\ddag Research supported in part by a generous grant from the Institute of Mathematics and its Applications.}

\begin{document}
\begin{abstract}
Chen, Kitaev, M\"{u}tze, and Sun recently introduced the notion of universal partial words, a generalization of universal words and de Bruijn sequences. Universal partial words allow for a wild-card character $\diam$, which is a placeholder for any letter in the alphabet. We extend results from the original paper and develop additional proof techniques to study these objects. For non-binary alphabets, we show that universal partial words have periodic $\diam$ structure and are cyclic, and we give number-theoretic conditions on the existence of universal partial words. In addition, we provide an explicit construction for an infinite family of universal partial words over non-binary alphabets. 
\end{abstract}

\maketitle
\section{Introduction}

\emph{Universal cycles} of a wide variety of combinatorial structures have been well studied~\cite{Chung}.
The best known examples are the \emph{de Bruijn sequences}, cyclic sequences over an alphabet $A$ that contain each word of length $n$ as a substring exactly once. For example, $(00010111)$ is a de Bruijn sequence for $A=\{0,1\}$ and $n=3$.

De Bruijn sequences can also be written as non-cyclic sequences. We denote the set of all words of length $n$ over a finite alphabet $A$ by $A^n$. A \emph{universal word} for $A^n$ is a word $w$ such that each word in $A^n$ appears exactly once as a substring of $w$. For example, $0001011100$, $0010111000$, and $1011100010$ are universal words for $\{0,1\}^3$, obtained by splitting the de Bruijn sequence $(00010111)$ and repeating the first $n-1=2$ characters at the end so that the same substrings occur when read non-cyclically. Every universal word corresponds to a de Bruijn sequence in this way.

The length of a universal word for $A^n$ is $|A|^n + n - 1$. It is known that universal words for $A^n$ exist for any $n$; they can be found through Eulerian and Hamiltonian cycles in the de Bruijn graph \cite{CKSpub}. However, as there are $|A|^{|A|^n+n-1}$ words over $A$ of the correct length, a brute-force search for universal words would quickly become intractable.

Now consider extending the alphabet $A$ to $A\cup\{\diam\}$, where $\diam \notin A$ is a wild-card character that can correspond to any letter of $A$. \emph{Partial words} are sequences of characters from $A\cup\{\diam\}$. Partial words are natural objects in coding theory and theoretical computer science. There are also applications in molecular biology and data communication \cite{Brownstein}. For example, when representing DNA and RNA as a string for computing purposes, the $\diam$ character can take the place of any unknown nucleotide.

For any partial word $u$, we denote by $u_i$ the $i$th character of $u$. Given partial words $u$ and $v$, we say that $u \subset v$ (or $u$ is a \emph{factor} of $v$, or $v$ \emph{covers} $u$), if $u$ can be found as a consecutive substring of $v$ after possibly replacing some $\diam$ characters in $v$ with a letter from $A$. Formally, $u \subset v$ if there exists $i$ such that $u_j = v_{i+j}$ for $1 \le j \le |u|$ whenever $v_{i+j} \in A$. For example, $01\diam 1 \subset 1\diam1\diam110$, but $01\diam1 \not\subset 1\diam10110$. 

Chen, Kitaev, M\"{u}tze, and Sun~\cite{CKSpub} introduced the notion of \emph{universal partial words} for $A^n$, generalizing universal words.

\begin{definition}A \emph{universal partial word} for $A^n$ is a partial word $w$ that covers each word in $A^n$ exactly once.
\end{definition}

For example, $\diam\diam0111$ is a universal partial word for $\{0,1\}^3$. Allowing $\diam$'s decreases the number of characters required to cover all words of length $n$, so universal partial words may be useful in questions related to storing information in compact form. Trivially, $\diam^n$ is a universal partial word for $A^n$ for any $A$ and $n$. A universal partial word $w$ is called \emph{trivial} if all of its characters are diamonds ($w =\diam^n$) or none are (i.e. $w$ is a universal word for $A^n$). We are interested in the existence of nontrivial universal partial words.

To begin, we present some preliminaries in Section \ref{prelim}. In Section \ref{binary}, we generalize the single-diamond study of \cite{CKSpub} to single strings of diamonds, and prove the following result:
\begin{theoremsp2}
For $|A|\geq 2$ and $2\le k \le n/2$, there does not exist a universal partial word $w= u\diam^k v$ where $u$ and $v$ are (possibly empty) words.
\end{theoremsp2}

Every universal word can be made into a de Bruijn sequence by deleting the last $n-1$ characters, which are the same as the first $n-1$ characters. We call a universal partial word \emph{cyclic} if its first and last $n-1$ characters are the same and non-overlapping, so that deleting the last $n-1$ characters and then reading cyclically yields the same substrings. In this sense, every universal word is cyclic, but some universal partial words over $\{0,1\}$, such as $\diam\diam0111$, are not. We prove in Section \ref{structure} that all non-binary universal partial words are cyclic.

\begin{theoremsp3} If $w$ is a nontrivial, non-binary universal partial word for $A^n$, then $w$ is cyclic.
\end{theoremsp3}

The possible diamond structures of non-binary universal partial words are more limited than when considering the binary case, which both restricts the possible lengths of universal partial words and gives rise to number theoretic conditions limiting their existence.

In Section \ref{construction}, we show that there are no nontrivial universal partial words for $n\leq 3$ over non-binary alphabets. 
For $n=4$, there are no universal partial words over alphabets of odd size. However, we show the following:

\begin{theoremsp5}
For any alphabet $A$ of even size, there exists a nontrivial universal partial word for $A^4$.
\end{theoremsp5}

In Section \ref{construction}, we give an explicit construction for words of this type. They are the first known nontrivial, non-binary universal partial words. Finally, in Section \ref{open}, we discuss some open questions.

\section{Preliminaries}\label{prelim}

An \emph{alphabet} $A$ is a set of symbols, which we call \emph{letters}. A \emph{character} is a letter or $\diam$. Throughout this paper, we will denote the size of the alphabet $A$ by $a$ and assume without loss of generality that the alphabet is $\{0, 1, \ldots, a-1\}$.

A \emph{word} over an alphabet $A$ is a sequence of letters. A \emph{partial word} is a sequence of characters from $A\cup\{\diam\}$. Note that a word cannot contain a $\diam$; we sometimes use the term \emph{total word} instead of \emph{word} for emphasis.

Note that if $w$ is a universal partial word for $A^n$, then the reverse of $w$ is as well. We will refer to this as the \emph{reversal property}. Permuting the letters of $A$ in $w$ also yields a universal partial word.

Given a universal partial word $w$, a \emph{window} of $w$ is a string of $n$ consecutive characters in $w$. The \emph{frame} of a partial word is the word over $\{\_,\diam\}$ obtained by replacing all letters from $A$ by the ``$\_$" character. For example, the partial word $x y \diam z \diam$ has frame $\_~\_~\diam~\_~\diam$. A \emph{window frame} is the frame of a window. In the above example, if $n=3$, the second window is $y\diam z$, and the second window frame is $\_~\diam~\_$.

Theorem \ref{diamondicity} shows that when $a \ge 3$, every window has the same number of diamonds, so for a universal partial word $w$ over a non-binary alphabet, we define the \emph{diamondicity} of $w$ as the number of diamonds that appear in any window of $w$. Furthermore, some universal partial words over binary alphabets also have well-defined diamondicity.

We also use the ideas of \emph{borders} and \emph{periods}, which are related and fundamental concepts in the study of combinatorics on words~\cite{BCK}. A partial word $w$ has a border $x$ of length $k$ if both the first $k$ characters and the last $k$ characters of $w$ cover $x$. A period of a word $w$ is a positive integer $p$ such that $w_i = w_j$ whenever $i \equiv j \pmod p$. Borders and periods have the following well known relationship.

\begin{theorem} [Folklore, e.g. Proposition 1 from \cite{Blanchet-Sadri}] \label{borderperiod} A word $w$ has period $p \in [|w|-1]$ 
if and only if it has a border of length $|w|-p$.\end{theorem}

\section{Consecutive Diamonds over a Binary Alphabet}\label{binary}

In \cite{CKSpub}, Chen et al.\ explore existence and non-existence of  universal partial words containing a single $\diam$. They then generalize their techniques to consider all universal partial words containing two $\diam$'s. In this section, we show non-existence results for a different type of generalization, that of universal partial words containing a single string of diamonds of length less than $n/2$.

We use the following lemma, which restricts the periodicity of any $n-k$ letters following $k$ consecutive diamonds in a universal partial word.

\begin{lemma} \label{noperiodicsuffix}
For $a\geq 2$, there does not exist a universal partial word $w$ containing the substring $u\diam ^k v$ where $u$ and $v$ are words, $|v|=n-k$, $v$ has period $p \leq k$, and $|u|=p$.
\end{lemma}

\begin{proof}
Suppose such a $w$ exists and let $v=v_1v_2\cdots v_{n-k}$, with $v_i\in A$. 
Since $v$ has period $p$, by Theorem \ref{borderperiod}, $v_1v_2\cdots v_{n-k-p}=v_{p+1}v_{p+2}\cdots v_{n-k}$, and so $\diam^p\,v_1v_2\cdots v_{n-k-p}$ covers $v$. 
Thus, the word $u\,0^{k-p}\,v\in A^n$ is covered by both the window starting with $u$ ($u\diam ^k v_1v_2\cdots v_{n-k-p}$) and the window starting with the first diamond ($\diam ^k v$). Therefore, $w$ is not a universal partial word.
\end{proof}

Lemma \ref{noperiodicsuffix} and the following result hold for any alphabet size. We address only the binary case in the following theorem, as a stronger non-binary analogue appears in Section \ref{structure} (Proposition \ref{nbkdiam}).

\begin{theorem} \label{kmiddlediamonds}
If $w = u\diam^kv$ is a universal partial word where $a = 2$, $k \ge 2$, and $u$ and $v$ are nonempty words, then $\vert u\vert \le n-1$ and $\vert v\vert \le n-1$.
\end{theorem}
\begin{proof}
Suppose to the contrary that $\vert u\vert \ge n$ or $\vert v\vert \ge n$. By the reversal property, without loss of generality, we may assume $\vert v\vert \ge n$.

There are $2^k$ words of length $n$ that begin $v_{1}v_{2}\cdots v_{n-k}$, all of which must be covered by $w$. 
One of these words is $v_{1}v_{2}\cdots v_{n-k}\cdots v_{n}$. 

Since $\diam^{k}v$ covers all words of length $n$ ending in $v_{1}v_{2}\cdots v_{n-k}$, the word $v_{1}v_{2}\cdots v_{n-k}$ is not covered elsewhere in $w$ with $k$ characters preceding it, so the only other places it can begin in $w$ are the first $k$ positions. Thus, the remaining $2^{k}\!-\!1$ words beginning with $v_{1}v_{2}\cdots v_{n-k}$ must be covered by windows that begin in the first $k$ positions of $w$. 

For each such position $i \in [k]$, let $N_{i}$ be the number of words in $A^n$ beginning with $v_{1}v_{2}\cdots v_{n-k}$ that are covered by $W_i := w_i w_{i+1} \cdots w_{i+n-1}$, the window beginning at position $i$ of $w$. Then $\sum_{i=1}^{k}N_{i}=2^{k}-1$, and for all $i$, $N_{i}$ is $0$ or a non-negative power of $2$. Specifically, if $j$ is the number of $\diam$'s in the last $k$ characters of $W_i$, then $N_i=2^j$ when the first $n-k$ characters of $W_i$ cover $v_{1}v_{2}\cdots v_{n-k}$ and $N_i=0$  otherwise. Since $N_i \leq \sum_{i=1}^k N_i=2^k-1$, we have $j< k$.

Now, $2^k-1$ can be uniquely written as a sum of powers of $2$ as $\sum_{i=0}^{k-1}{2^{i}}$, so we must have that $N_{1},\dots,N_{k}$ are $2^{0},2^{1},\dots,2^{k-1}$ in some order.
Since $N_{i+1}\neq N_i$ and the number of $\diam$'s in $W_i$ differs by at most one from the number of  $\diam$'s in $W_{i+1}$, we have $N_{i+1}=2N_{i}$ or $N_{i+1}=\frac{1}{2}N_i$ for all $i\in [k-1]$. Also, since we use each power exactly once, the sequence of $N_i$'s must be monotonic.

\begin{enumerate}
\item[Case 1:] The $N_i$'s are increasing, so $N_{i}=2^{i-1}$ for all $i \in [k]$, and\\ $w=v_{1}v_{2}\cdots v_{n-k}w_{n-k+1}w_{n-k+2}\cdots w_{n}\diam^{k}v$.\\
Since $N_{i}>0$ for all $i$, we have $$v_1v_2\cdots v_{n-k}=v_2\cdots v_{n-k}w_{n-k+1}=\cdots = w_k\cdots w_{n-1},  $$ so $v_{1}=v_{2}=\cdots=v_{n-k} = w_{n-k+1} = w_{n-k+2} = \cdots =w_{n-1}$. In particular, $v_{1}v_{2}\cdots v_{n-k}$ has period 1. Then, since $\vert u\vert\geq 1$, by Lemma \ref{noperiodicsuffix}, $w$ is not a universal partial word.

\item[Case 2:]  The $N_i$'s are decreasing, so $N_{i}=2^{k-i}$ for all $i$, and $w=v_1v_2\cdots v_{n-k-1}\diam^{k}v$ so that there are $2^{k-1}$ words in $A^n$ starting with $v_1\cdots v_{n-k}$ in $W_1 = v_1v_2 \cdots v_{n-k-1}\diam^k v_1$, $2^{k-2}$ words in $A^n$ starting with $v_1\cdots v_{n-k}$ in $W_2 = v_2 \cdots v_{n-k-1}\diam^k v_1 v_2$, and so on through $2^0 = 1$ word in $A^n$ starting with $v_1\cdots v_{n-k}$ in $$W_k = \begin{cases}v_{k}\cdots v_{n-k-1} \diamond^k v_1\cdots v_k& k\leq n-k-1 \\
\diamond^{n-k} v_1\cdots v_k & k\ge n-k\end{cases}.$$
Since $\vert u\vert\geq 1$, this case is only possible if $n-k-1 \geq 1$.
Since $N_{i}>0$ for all $i$, $W_1$ and $W_2$ both cover $v_1\cdots v_{n-k}$ with their first $n-k$ characters, so $v_{1}=v_{2}=\cdots=v_{n-k-1}$. 
If $n-k-1\geq 2$, then $W_1$ and $W_2$ both cover
the word $v_1v_2\cdots v_{n-k}0^{k-2}v_1v_2$.

Thus, $n-k-1=1$, and $w=v_1\diam^kv=v_1\diam^{n-2}v$. 
Without loss of generality, assume $v_1=0$. 
Then, $w=0\diam^{n-2}0v_2v_3\cdots v_{\ell}$ with $\ell\geq n$. 
Now, $W_1$ covers all words beginning and ending with $0$. 
Thus, $v_n=1$, and every letter $n-1$ positions after a diamond is a $1$, so $w=0\diam^{n-2}01^{n-1}v_{n+1}\cdots v_{\ell}$. If $v_{n+1}$ did not exist, then $1^n$ would not covered, and $w$ would not be a universal partial word. 
Next, since $w$ covers $1^n$, $v_{n+1}=1$. Otherwise, $01^{n-1}$ would be covered twice in $w$. 
Thus, $w=0\diam^{n-2}01^nv_{n+2}\cdots v_{\ell}$ (note that $v_{n+2}$ must exist for $w$ to cover $1^{n-1}0$). 
Since $W_2=\diam^{n-2}01$, the string $01$ cannot appear elsewhere in $w$ with $k$ characters preceding it. 
Thus, $v_{n+2}=v_{n+3}=\cdots=v_{\ell}=0$. 
Since $0^n$ is covered by $W_1$, $w=0\diam^{n-2}01^n0^p$ where $p\leq n-1$. 
Now, if $n$ is even, $(10)^\frac{n}{2}$ is not covered by $w$, and if $n$ is odd, $(10)^\frac{n-1}{2}0$ is not covered by $w$. Thus, $w$ is not a universal partial word. 
\end{enumerate}
Therefore, no such $w$ exists.
\end{proof}

We will also require the following lemma.

\begin{lemma}\label{lem:kdiamondstart}
For $a\geq 2$ and $2\le k \le n/2$, there does not exist a universal partial word $w=\diam^k v$ where $v$ is a word.
\end{lemma}

\begin{proof}
If $|v| \le n-k-1$, then $|w| \le n-1$, so $w$ covers no words of length $n$ and is not a universal partial word. Therefore we may assume $|v| \ge n-k$.

Let  $v' = v_1 \dots v_{n-k}$. All words ending with $v'$ are covered by $\diam^k v'$, so $v'$ cannot be covered later in $v$. Thus the $a^k$ words beginning with $v'$ must be covered by the first $k+1$ windows of $w$.
Since $|v'|\geq |\diam^k|$, in each of these windows the letter covered by each $\diamond$ is fixed by its corresponding letter in $v'$. Thus, only one word of length $n$ can be covered by each of the $k+1$ windows, and so $w$ contains at most $k+1$ words beginning with $v'$. Therefore, $w$ cannot contain all $a^k\geq 2^k > k+1$ of them, and $w$ is not a universal partial word.
\end{proof}

Together, these results lead to the following theorem.

\begin{theorem}\label{thm:kdiamond}
For $a\geq 2$ and $2\le k \le n/2$, there does not exist a universal partial word $w= u\diam^k v$ where $u$ and $v$ are (possibly empty) words.
\end{theorem}

	\begin{proof} Proposition \ref{nbkdiam} gives a stronger result for $a \ge 3$, so here we assume $a=2$. We proceed by contradiction. Assume  $w= u\diam^k v$ is a universal partial word with $2 \leq k \leq n/2$ such that $u$ and $v$ do not contain any $\diam$'s.  By Theorem \ref{kmiddlediamonds} and Lemma \ref{lem:kdiamondstart}, we have $1\leq \vert u \vert, \vert v \vert \leq n-1$. There are at most $a^k$ words in $A^n$ covered at each possible starting position in $w$. Since $|w|=\vert u \vert + \vert v \vert + k$, the number of words covered by $w$ is at most $\left(\vert u \vert + \vert v \vert + k - (n-1)\right)a^k$. Since $|u|,|v|\leq n-1$ and $k\leq n/2$,
	$$\left(\vert u \vert + \vert v \vert + k - (n-1)\right)a^k\leq (n-1+k)a^k \leq \left((3/2)n - 1\right)a^{n/2}.$$
	
	For $n \ge 7$, this is less than $a^n$, so $w$ does not cover every word in $A^n$ and is not a universal partial word.
	
	Now, since $2\le k\le n/2$, we only need to consider the combinations $n=6$ with $k=2$ or $3$, $n=5$ with $k=2$, and $n=4$ with $k=2$. For $k=2$, there are at most $a$ words covered by $u$ and the first $\diam$. There are at most $(n-1)a^2$ words covered using both $\diam$'s (starting at each of the last $n-2$ positions of $u$ and at the first $\diam$). Finally, there are at most $a$ words covered by the second diamond and $v$. Since $|u|, |v| \leq n-1$, no word can end before or begin after the $\diam$'s. Thus, there are at most $2a+(n-1)a^2$ words covered by $w$. Now, $2a+(n-1)a^2<a+na^2$ since $a^2>a$ for $a\geq 2$. If $n \geq 5$, $w$ does not cover all of the $a^n$ words in $A^n$.
	
	For $k=2$ and $n=4$, we have $1 \le |u|, |v| \le n-1 = 3$, and $w$ must cover $2^4=16$ words of length $4$. 
If $|u|=1$ or $|v|=1$, then $w$ covers at most $2^2+2^2+2=10$ words.
Thus, $|u|,|v|\ge 2$, and $w$ contains $st\diam\diam xy$ for some $s,t,x,y\in A$. Here $w$ covers $stxy$ twice, so $w$ is not a universal partial word and no such $w$ exists.

	The remaining case is $n=6$, $k=3$. Without loss of generality $\vert u \vert = \vert v \vert = n-1$, since shortening $u$ and $v$ decreases the number of words covered. In this case, $w=u_1u_2u_3u_4u_5\diam\diam\diam v_1v_2v_3v_4v_5$. Considering each window of $w$, the number of words covered by $w$ is at most $2+2^2+2^3+2^3+2^3+2^3+2^2+2 = 44 < 2^6$, so $w$ does not cover every word in $A^n$ and is not a universal partial word.
	\end{proof}

This result illustrates that even in the binary case, where universal partial words have less rigid structure, the presence of even two consecutive diamonds in a universal partial word would force the appearance of other diamonds.

\section{Structural Conditions over Non-Binary Alphabets}\label{structure}

In \cite{CKSpub}, Chen et al.\ show that cyclic universal partial words have a rigid diamond structure and length. We will show in this section that all non-binary universal partial words have these properties and that these are enough to prove that all non-binary universal partial words are cyclic. This is in contrast to the binary case, where non-cyclic universal partial words are common. For example, $\diam\diam0111$ and $\diam001011\diam$ are non-cyclic universal partial words for $\{0,1\}^3$.

The following result provides periodic structure for the frame of a universal partial word over a non-binary alphabet.

\begin{theorem}\label{diamondicity} Let $w$ be a universal partial word for $A^n$ with $a \ge 3$. If $w_i =\diam$, then $w_j =\diam$ for all $j \equiv i \pmod n$. In other words, the frame of any such universal partial word has period $n$.
\end{theorem}

\begin{proof}We will show that if $w_i =\diam$, then $w_{i+n} =\diam$ for any $i \in [|w|-n]$. By the reversal property, this is sufficient to obtain the theorem.

Assume $w_i =\diam$, and suppose to the contrary that $w_{i+n} \in A$. Without loss of generality, let $w_{i+n} = 0$. Let $v$ be any total word covered by $w_{i+1}w_{i+2}\cdots w_{i+n-1}$. Then, the window of $w$ starting at $w_{i+1}$ covers $v0$, and the window starting at $w_i$ covers $\diam v$ and thus covers all words of length $n$ ending in $v$. Thus, to avoid repeating words ending in $v$, $v$ cannot be covered elsewhere in $w$ except by the first window. 

Since $w$ is a universal partial word, it must also cover the words $v1$ and $v2$ exactly once. These words must be covered by $w_1\cdots w_n$. For both $v1$ and $v2$ to be covered, $w_n=\diam$, so $w_1\cdots w_n$ also covers the word $v0$. This instance of the word $v0$ is different from the one that appears starting with $w_{i+1}$. Therefore the word $v0$ is covered twice in $w$, contradicting that $w$ is a universal partial word.
\end{proof}

Theorem \ref{diamondicity} makes possible the following definition of diamondicity. 
\begin{definition}
For $w$ a universal partial word for $A^n$, with $a\geq 3$, the \emph{diamondicity} of $w$ is the number of diamonds in each window. The window frames of $w$ are cyclic shifts of the first window frame.
\end{definition}

For binary alphabets, the number of diamonds is not enough to determine the length of a universal partial word. For example, $\diam\diam 0111$ and $\diam 001011\diam$ both are universal partial words for $\{0,1\}^3$ with two diamonds but have different lengths, 6 and 8. However, universal partial words with diamondicity have length determined by $a$, $n$, and $d$. Corollary 15 from \cite{CKSpub} determines the lengths of cyclic universal partial words, but the corollary below requires only the weaker assumption of well-defined diamondicity. We note that diamondicity is also well-defined for some universal partial words over binary alphabets; for example, $01 \diam 110 \diam 001 \diam$ is a universal partial word for $\lbrace 0,1 \rbrace^4$ with $d=1$. 

\begin{corollary}\label{length}
If $w$ is a universal partial word for $A^n$ with diamondicity $d$, then $\vert w \vert = a^{n-d}+n-1$.
\end{corollary}
\begin{proof}
Each window contains $d$ diamonds, by Theorem \ref{diamondicity}. Therefore, each window covers $a^d$ words. Since $w$ is a universal partial word, it must cover all $a^n$ words exactly once, so $w$ contains $\frac{a^n}{a^d}=a^{n-d}$ windows. The last window contains $n-1$ characters that do not themselves start windows, so $\vert w \vert = a^{n-d}+n-1.$ 
\end{proof}
 
Theorem \ref{diamondicity} shows that the correct parameter to consider in the non-binary case is not the number of diamonds, but rather the density of diamonds (diamondicity), since there are roughly $(d/n) \cdot a^{n-d}$ diamonds in a universal partial word over a non-binary alphabet. 

Recall that all universal words are cyclic, but not all universal partial words are. We will use the following lemma to prove that all universal partial words over non-binary alphabets are cyclic in Theorem \ref{cyclicity}.

\begin{lemma}\label{pseudocyclicity} If $w$ is a universal partial word for $A^n$ with $a \ge 3$, then the first $n-1$ characters of $w$ equal the last $n-1$ characters of $w$. In other words, as a total word over the extended alphabet $A \cup \{\diam\}$, $w$ has a border of length $n-1$.\end{lemma} 

\begin{proof}
Let $w$ be a universal partial word for $A^n$ with $a \ge 3$. Let $v$ be a word of length $n-1$ covered by $w_{1}\cdots w_{n-1}$. We show in two cases that the last $n-1$ characters of $w$ cover $v$.

First, if $w_n =\diam$, all $a$ words beginning with $v$ are covered by $w_{1}\cdots w_{n}$. If $v$ were covered elsewhere in $w$, except for the end of $w$, then that string and the character immediately following it would cover a word beginning with $v$ also covered by $w_{1}\cdots w_{n}$. Therefore, for $w$ to cover the $a$ words ending with $v$, the last $n-1$ characters of $w$ cover $v$.

On the other hand, if $w_n \neq\diam$, exactly one word beginning with $v$ is covered by $w_{1}\cdots w_{n}$. The remaining $a-1$ words beginning with $v$ must be covered by strings of $w$ that are not immediately followed by $\diam$, as this would duplicate $v_{1}v_{2}\cdots v_{n-1}w_{n}$. These words are covered by exactly $a-1$ other strings, $w_i\cdots w_{i+n-2}$ (for $i \in I$ with $|I| = a-1$),
that cover $v$ and are followed by a letter in $A$. By Theorem \ref{diamondicity}, these strings cannot be preceded by $\diam$, and thus the $w_{i-1}w_i\cdots w_{i+n-2}$'s
cover $a-1$ distinct words ending with $v$. If $|w|-n+2 \in I$, then one of these $a-1$ strings is at the end of $w$; i.e., the last $n-1$ characters of $w$ cover $v$. Otherwise, the remaining word ending with $v$ (out of $a$ total) must be covered by the last $n$ characters of $w$ so as not to duplicate a word starting with $v$, so the last $n-1$ characters of $w$ cover $v$.

Thus, the last $n-1$ characters of $w$ cover all words covered by the first $n-1$ characters, and by the reversal property, the first $n-1$ characters cover all words covered by the last $n-1$ characters. Hence the first $n-1$ characters equal the last $n-1$ characters.
\end{proof}

Lemma \ref{pseudocyclicity} does not guarantee that a universal partial word is cyclic, because the last $n-1$ characters might overlap with the first $n-1$ characters if the word is short enough. We call a partial word \emph{pseudocyclic} if its first and last $n-1$ characters are the same, whether or not they overlap. Pseudocyclicity is sufficient to show that the frame of $w$ has period $n$.  Chen et al.\ proved this for cyclic universal partial words (Lemma 14 in \cite{CKSpub}), but their proof techniques extend to the case of pseudocyclicity. We present an alternate proof of this result.  
\begin{prop}\label{everythingisdiamondy} If $w$ is a pseudocyclic universal partial word, then $w$ has well-defined diamondicity.\end{prop}

\begin{proof}If $a \ge 3$, then $w$ has well-defined diamondicity by Theorem \ref{diamondicity}. Let $w$ be a pseudocyclic binary universal partial word with $w_i = \diam$ for some $i \in [|w|-n]$.  Let $v$ be a word in $\{0,1\}^{n-1}$ covered by $w_{i+1}\cdots w_{i+n-1}$.
	
If $w_{i+n} \neq \diam$, without loss of generality we assume $w_{i+n} = 0$. The word $v1$ must be covered in $w$, so $v$ is covered elsewhere in $w$. If $v$ is covered at the beginning of $w$, then $v$ is also covered at the end of $w$ since $w$ is pseudocyclic, and this instance of $v$ at the end of $w$ is not the same as the one covered by $w_{i+i}\cdots w_{i+n-1}$ because $|w| \ge i+n$. Therefore there is some other instance of $v$ with a character preceding it, and either the word $0v$ or the word $1v$ is covered twice in $w$. But $w$ is a universal partial word, so $w_{i+n} = \diam$.
\end{proof}

We will use a few intermediate results including the following lemma to prove Theorem \ref{cyclicity}, that non-binary universal partial words are cyclic.

\begin{lemma}\label{lem:frameshifts}For any pseudocyclic universal partial word $w$, all distinct cyclic shifts of the first window frame of $w$ appear the same number of times in $w$.
\end{lemma}
\begin{proof}Let $f=f_1$ be the first window frame of $w$. Let $\{f_i\}_{i=1}^\infty$ be the repeating series of cyclic shifts of $f$, so $f_i$ is the result of shifting $f_1$ to the left, cyclically, $i-1$ times. Then the $i$th window frame of $w$ is $f_i$ for $i \in [a^{n-d}]$, since $a^{n-d}$ is the number of windows in $w$. Let $f'$ be the shortest word such that $f = (f')^s$ for some $s \in \mathbb N$, and let $m = |f'|$. Then $f_i = f_j$ if and only if $i \equiv j \mod m$. By pseudocyclicity, the first and last $n-1$ characters of $w$ have the same frame. Thus shifting the last window frame of $w$, $f_{a^{n-d}}$, one more time yields $f_{a^{n-d}+1} = f_1$, restarting the cycle, so no new cyclic shifts of $f$ can appear in $\{f_i\}_{i=a^{n-d}+1}^\infty$. That is, all of the distinct cyclic shifts of $f$ appear as window frames in $w$, and there are $m$ of them, $f_1, f_2, \ldots, f_m$, by the minimality of $m$. Furthermore, since the first window frame in $w$ is $f_1$ and the last is $f_{a^{n-d}}=f_m$ (note $m \equiv a^{n-d} \mod m$ by Theorem \ref{divisibility}), the cycle $f_1, f_2, \ldots, f_m$ of all the distinct cyclic shifts of $f_1$ is repeated some integer number of times in $w$.
\end{proof}

Next, we will use this lemma to prove a bound on diamondicity. By Corollary \ref{length}, the higher the diamondicity of a word, the shorter the word. However, as diamondicity increases, it becomes harder to avoid covering words multiple times. In fact, it is possible to bound the potential diamondicities of a universal partial word in terms of $n$ as shown in the following proposition.

\begin{prop}\label{diambound}
For every $k$, if $n \geq k(k-1)+2$, then there does not exist a pseudocyclic universal partial word for $A^n$ with diamondicity $d \ge n-k$. In particular, if $w$ is a pseudocyclic universal partial word for $A^n$
, then $d < n-\sqrt{n-\frac{7}{4}} - \frac{1}{2}$.
\end{prop}
\begin{proof} For a given $k$ and given $n \ge k(k-1)+2$, suppose $w$ is a universal partial word with diamondicity $n-k$. Consider the first window frame $f_1$ of $w$. Let $x_1,x_2,\dots, x_k \in [n]$ be the positions of the $\_$'s in $f_1$.

There are at most $k(k-1)$
distinct distances between pairs $x_r < x_s$, since each pair has distances $\ell = x_s - x_r$ and $n-\ell$ between $x_r$ and $x_s$ cyclically. Let $\mathcal{L} \subseteq [n-1]$ be the set of these cyclic distances between $x$'s. 

We will use the notation $\{f_i\}_{i=1}^\infty$ for the cyclic shifts of the first frame $f_1$ of $w$ as in Lemma \ref{lem:frameshifts}.
For $i \in [n]$, if $f_i$ has $\_$ in position $x_r$ for some $r$, then there is an $s$ such that $x_r + (i-1) \equiv x_s \mod n$, so $i-1 \in \mathcal{L}$.

Since $n \geq k(k-1)+2$, we have $n-1 > |\mathcal{L}|$, so there is at least one $i$, say $i'$, such that $i'-1 \in [n-1]$ and $i'-1 \notin \mathcal{L}$. For this $i' \in [n]$, $f_{i'}$ shares no $\_$ positions with $f_1$.

Let $j \in [m]$ and $j \equiv i' \mod m$. Then $f_{i'} = f_j$, and $w$ has at least $j$ window frames, so by Lemma \ref{lem:frameshifts} the $j$th window frame exists and shares no $\_$ positions with $f_1$. The first and $j$th windows both cover some word of length $n$ since at every position at least one of them has a $\diam$. This contradicts that $w$ is a universal partial word.

For given $n$ and $k$, increasing the diamondicity to $d > n-k$ can only increase the number of window frames that share no $\_$ positions with $f$. Thus if $d \ge n-k$, $w$ is not a universal partial word.

To show that if $w$ is a universal partial word for $A^n$, then $d < n-\sqrt{n-\frac{7}{4}} - \frac{1}{2}$, let $k = \sqrt{n-\frac74}+\frac12$. Solving this equation for $n$, we obtain $n = k(k-1)+ 2$. Applying the first part of the theorem, we have $d < n - k = n - \sqrt{n-\frac74}-\frac12$.
\end{proof}

For $n=4$, Proposition \ref{diambound} shows every nontrivial non-binary universal partial word has diamondicity $d = 1$. In addition, Theorem 19 from \cite{CKSpub} also follows as a quick corollary.

This proof method is insufficient for giving a fractional bound on diamondicity, since it is possible to construct window frames for a given density for large enough $n$ that do not clearly cover any words twice.

While we suspect a stronger diamondicity bound is possible, this one is sufficient to show that all nontrivial universal partial words over non-binary alphabets are cyclic in the classical sense.

\begin{theorem}\label{cyclicity}If $w$ is a nontrivial, non-binary universal partial word for $A^n$, then $w$ is cyclic.\end{theorem}

\begin{proof}By Lemma \ref{pseudocyclicity}, the first and last $n-1$ characters of $w$ are the same. We need only ensure that $w$ is long enough to prevent these first and last $n-1$ characters from overlapping.

By Proposition \ref{diambound}, we have that $d < n - \sqrt{n - \frac{7}{4}}-\frac{1}{2}$. Therefore
$$
|w| = a^{n-d} + n - 1 > a^{\sqrt{n - 7/4} + 1/2} + n - 1 > 2n-2
$$
since $a \geq 3$ by hypothesis and $n \ge 4$ by Proposition \ref{prop:smallwordsafterall}. Thus the first and last $n-1$ characters of $w$ do not overlap. \end{proof}

Words over binary alphabets containing a single $\diam$ or a single string of consecutive $\diam$'s were studied in \cite{CKSpub}. In fact, over non-binary alphabets, such words do not exist. The following is a non-binary analogue and extension of Theorem \ref{kmiddlediamonds}.

\begin{prop}\label{nbkdiam}For $a \geq 3$, there does not exist a universal partial word $w = u \diam^k v$ where $u$ and $v$ are (possibly empty) words and $1 \leq k \le n-1$.\end{prop}

\begin{proof}
For $a \geq 3$, let $w = u\diam^k v$ where $u$ and $v$ are words. By Theorem \ref{diamondicity}, we know that $|u|, |v| \leq n-k$. By Corollary \ref{length}, we know that $|w| = a^{n-k} + n - 1$. Therefore
\begin{align*}
a^{n-k} + n - 1 &= |u| + |v| + k \leq n-k + n-k + k = 2n - k
\end{align*}
and so $a^{n-k} \leq n - k + 1$. This is a contradiction when $a \geq 3$ and $k \le n-1$.
\end{proof}

Note that for $k=0$ and $k=n$, there are trivial universal partial words of this form. Theorem 5 in \cite{CKSpub} also proves the $k=1$ case.

We would like to be able to show the nonexistence of universal partial words based only on the parameters $a,n,$ and $d$. We will take advantage of the cyclic nature of the window frames of a partial word. The following lemma applies to any word but will be used in the context of frames.

\begin{lemma}\label{cyclicshift}
Cyclically shifting a word $f$ $i$ times yields $f$ if and only if there is a word $f'$ such that $f = (f')^s$ for some $s \in \mathbb N$, where $|f'| = i$.
\end{lemma}

\begin{proof}
If $f = (f')^s$ for some $s \in \mathbb N$, where $|f'| = i$, then cyclically shifting $f = (f') (f')^{s-1}$ $i$ times yields $(f')^{s-1} (f') = (f')^s = f$. For the reverse direction, suppose cyclically shifting $f=f_1f_2\cdots f_n$ $i$ times yields $f$. Then $f_1 f_2 \cdots f_n = f_{i+1} f_{i+2} \cdots f_{n} f_1 f_2 \cdots f_i$. Let $f' = f_1 f_2 \cdots f_i$. Note $f$ has period $i$ (since $f_j = f_{i+j \pmod n}$ for all $j \in [n]$) and begins and ends with $f'$, so $i | n$ and $f = (f')^{n/i}$, where $n/i \in \mathbb N$.
\end{proof}

Using Lemma \ref{cyclicshift} and Theorem \ref{borderperiod}, we can prove the following theorem, which by Theorem \ref{cyclicity} applies to all non-binary universal partial words.
 
\begin{theorem}\label{divisibility} Let $w$ be a pseudocyclic universal partial word for $A^n$. If $f$ is the first window frame of $w$, and $i$ is the length of the shortest frame $f'$ such that $f = (f')^s$ for some $s \in \mathbb N$, then $i \big | \gcd(a^{n-d},n)$.
\end{theorem}

\begin{proof}First, we note that $si = |f| = n$, so $i \big | n$.

Next, we will show that $i \big | a^{n-d}$.

The length of $w$ is $N = a^{n-d} + n-1$. By pseudocyclicity and Theorem \ref{borderperiod}, $w$ considered as a word over $A\cup\{\diam\}$ has period $N - (n-1) = a^{n-d}$. In particular, the frame of $w$ has period $a^{n-d}$, so cyclically shifting the first window frame of $w$, $f$, $a^{n-d}$ times yields the same frame.

Let $r = a^{n-d} \pmod{i}$, so $r \in \{0, 1, \ldots, i-1\}$. If $r \in [i-1]$, then by Lemma \ref{cyclicshift} and the minimality of $i$, we can conclude that if we shift $f$ $r$ times, then we do not get the same frame. But shifting $f$ $r$ times must yield the same frame because by Lemma \ref{cyclicshift}, it is equivalent to shifting $f$ $a^{n-d}$ times, which yields the same frame. Therefore $r = 0$, i.e. $i | a^{n-d}$.\end{proof}

This gives rise to some immediate number-theoretic corollaries which allow us to eliminate many combinations of $a,n,$ and $d$. For example, if $\gcd(a,n)=1$, then there are no nontrivial pseudocyclic universal partial words for $A^n$ (Corollary 16 from \cite{CKSpub}). 

\begin{corollary} If $\gcd(a^{n-d},n)=2$, then there are no nontrivial universal partial words for $A^n$ with diamondicity $d$.\end{corollary} 

\begin{proof} If $i$ from Theorem \ref{divisibility} divides $2$ and $f$ contains both letters and $\diam$s, then $f'=\diam \_$ or $f' = \_ \diam$. Without loss of generality, let $f=\diam \_$. Let $v_1, v_2, \ldots, v_{n/2}$ be the letters of $f$. Then $w$ covers $v_1v_1v_2v_2\ldots v_{n/2}v_{n/2}$ twice.
\end{proof}

In particular, pseudocyclic binary universal partial words require that $n$ be a multiple of 4.

\begin{corollary}For $a\geq 3$, if $\gcd(a^{n-d},n) = p$ for some prime $p$ and there exists a universal partial word for these values of $a, n$, and $d$, then $d \in \{kn/p : k \in [p-1]\}$. \end{corollary}

\begin{proof} Assume $w$ is a universal partial word for $A^n$ with diamondicity $d$. Let $f$ be the first window frame of $w$, and let $f'$ be the shortest word such that $f=(f')^s$ for some $s\in \mathbb N$. By Theorem \ref{divisibility}, $|f'| = 1$ or $p$. Assuming $w$ is nontrivial, we have $|f'| = p$. Then, $n = ps$. Let $d'$ be the number of diamonds in $f'$, so $d = d's = d'n/p$. Note $d' \in [p-1]$ as $f'$ must have at least one letter and at least one diamond for $w$ to be nontrivial.\end{proof}

\section{Construction of a Universal Partial Word}\label{construction}
Given the results of Section~\ref{structure}, it is tempting to believe that universal partial words may not exist for non-binary alphabets. This is in fact the case for small $n$.
\begin{prop}\label{prop:smallwordsafterall}
For $a\geq 3$, there does not exist a nontrivial universal partial word $w$ for $n\leq 3$.
\end{prop}
\begin{proof}
It is clear that there is no nontrivial universal partial word for $n=1$. 

For $n=2$, assume $w$ is a nontrivial universal partial word. By Theorem \ref{diamondicity}, the diamondicity of $w$ is $d=1$, and by Corollary \ref{length}, $|w|=a+1\geq 4$, Now, $w=\diam x_2\diam x_4\diam \ldots$ or $w=x_1 \diam x_3 \diam x_5 \diam \ldots$. In the former case, the first three characters cover $x_2 x_2 $ twice. In the latter case, the second through fourth characters cover $x_3x_3$ twice. Therefore, $w$ is not a universal partial word. 

For $n=3$, assume $w$ is a nontrivial universal partial word.  By Theorem \ref{diamondicity}, $w$ must contain either the window frame $\diam \_ \_$ or the window frame $\diam\diam\_$. 

In the first case, consider the word $000$, which, without loss of generality, is not covered at the beginning or end of $w$. The partial word $w$ cannot contain the string $\diam00\diam$, since that would cover $000$ twice, so $w$ must contain the string $\diam x0\diam0y\diam$ in order to cover $000$. But this string covers $x0y$ twice. Therefore, $w$ is not a partial word.

In the second case, $\vert w \vert =a+2\geq 5$ and must contain $\diam\diam x\diam\diam, x\diam\diam y\diam$, or $\diam x\diam\diam y$. The first four characters of each of these partial words cover either $xxx$ or $xxy$ twice. Therefore, $w$ is not a universal partial word. 
\end{proof}

While these small $n$ are not fruitful, for $n=4$ not only are we able to find nontrivial examples, we can construct a family of universal partial words for any even alphabet size. Note that by Theorem \ref{divisibility}, there are no nontrivial universal partial words for $n=4$ when $a$ is odd, for then $\gcd(a,4)=1$.  

\begin{theorem}\label{thm:allinthefamily}
Let $a$ be even.
\begin{enumerate}

\item Construct the following sequence of $a^3/4$~letters:
\[
\underbrace{0,1,0,1,\ldots,0,1}_{a^2/2~\mathrm{letters}},
\underbrace{2,3,2,3,\ldots,2,3}_{a^2/2~\mathrm{letters}},\ldots,
\underbrace{a-2,a-1,a-2,a-1,\ldots,a-2,a-1}
	_{a^2/2~\mathrm{letters}}.
\]
Call this sequence~$\left<x_i\right>$, where $i \in [a^3/4]$.

\item Construct the following sequence of $a^2/2$~letters:
\[
\underbrace{0,1,0,1,\ldots,0,1}_{a~\mathrm{letters}},
\underbrace{2,3,2,3,\ldots,2,3}_{a~\mathrm{letters}},\ldots,
\underbrace{a-2,a-1,a-2,a-1,\ldots,a-2,a-1}
	_{a~\mathrm{letters}}.
\]
Repeat this sequence $a/2$~times to get a sequence of $a^3/4$~letters, and
call the resulting sequence~$\left<y_i\right>$, where $i \in [a^3/4]$.

\item Construct the following sequence of $a$~letters:
\[
1,0,3,2,5,4,\ldots,a-1,a-2.
\]
Repeat this sequence $a^2/4$~times to get a sequence of $a^3/4$~letters,
and call the resulting sequence~$\left<z_i\right>$, where $i \in [a^3/4]$.

\item For $i \in [a^3/4]$, let $w_i$ be the word $x_iy_iz_i$.

\item Take $u=w_1\diam w_2\diam\ldots\diam w_{a^3/4}\diam w_1$.

\end{enumerate}
Then $u$ is a universal partial word for $A^4$.
\end{theorem}
\begin{proof}
Each $w_i$ has length $3$, and there are $a^3/4$ diamonds, so $$\vert u \vert = 3(a^3/4+1)+a^3/4 = a^3 + 3.$$ This is the length of a universal partial word with $n=4$ and diamondicity $d=1$, so it is sufficient to show that no word is covered twice. 

Suppose $u$ covers $v_1 v_2 v_3 v_4$ twice. Let $v$ and $v'$ be the two windows of $u$ which cover $v_1 v_2 v_3 v_4$. Either $v$ and $v'$ have a diamond in the same position or in a different position.
\begin{enumerate}

\item[Case 1:] 

$v$ and $v'$ have the same frame.\\
Let us consider the situation where $v= x_i y_i z_i\diam$ and $v' = x_j y_j z_j\diam $ with $i< j$. Suppose $x_i = x_j = c$. Then we must have that 
$$ \frac {\lfloor c/2 \rfloor a^2}{2} < i,j \leq \frac {(\lfloor c/2 \rfloor +1) a^2}{2}.$$ Within this range, $y_i=y_j$ implies that $j-i< a$, but $z_i = z_j$ implies that $j-i\geq a$. This is a contradiction, so $v$ and $v'$ do not cover the same word.

Note that any possible placement of the diamond in $v$ and $v'$ (e.g. $v=z_i\diam x_{i+1} y_{i+1}$ and $v'=z_j\diam x_{j+1}y_{j+1})$ can be handled in a similar manner as the above situation, since our argument only makes use of the difference between the indices, and these differences remain the same.
\item[Case 2:] $v$ and $v'$ have diamonds in different positions.\\
From our construction, for all $i \in [a^3/4]$, we have $i \not\equiv x_i \equiv y_i \pmod 2$ and $i \equiv z_i \pmod 2$.
\begin{itemize}
\item If $v=x_iy_iz_i\diam$ and $v'=y_jz_j\diam x_{j+1}$, then $x_i = y_j$ and $y_i = z_j$, so $$j \not\equiv y_j = x_i \equiv y_i = z_j \equiv j \pmod 2,$$ a contradiction.
\item If $v=x_iy_iz_i\diam$ and $v'=z_j\diam x_{j+1}y_{j+1}$, then $x_i= z_j$ and $z_i= x_{j+1}$, so $$i \equiv z_i = x_{j+1} \equiv z_{j} = x_i \not\equiv i \pmod 2,$$ a contradiction.  
\item If $v=x_iy_iz_i\diam$ and $v'=\diam x_j y_j z_j$, then $z_i = y_j$ and $y_i = x_j$, so $$i \equiv z_i = y_j \equiv x_j = y_i \not\equiv i\pmod 2,$$ a contradiction.
\item If $v=y_iz_i\diam x_{i+1}$ and $v'=z_j\diam x_{j+1} y_{j+1}$, then $y_i=z_j$ and $y_{j+1}=x_{i+1}$, so $$i+1 \equiv y_i = z_j \equiv y_{j+1} = x_{i+1} \not\equiv i+1 \pmod 2,$$ a contradiction.
\item If $v=y_iz_i\diam x_{i+1}$ and $v'=\diam x_{j} y_{j} z_{j}$, then $z_i=x_{j}$ and $x_{i+1}=z_j$, so $$j \not\equiv x_j = z_i \equiv x_{i+1} = z_j \equiv j \pmod 2,$$ a contradiction.
\item If $v=z_i\diam x_{i+1} y_{i+1}$ and $v'= \diam x_j y_j z_j$, $x_{i+1} = y_j$ and $y_{i+1} = z_j$, so $$j \equiv z_j = y_{i+1} \equiv x_{i+1} = y_j \not\equiv j \pmod 2,$$ a contradiction.
\end{itemize}

\end{enumerate}
Therefore no word $v_1v_2v_3v_4$ is covered twice by $u$, and $u$ is a universal partial word.
\end{proof}
We illustrate this construction by giving an example for $a=4$.

\begin{example}
The string
\[
001\diam110\diam003\diam112\diam021\diam130\diam023\diam132\diam201\diam310\diam203\diam312\diam221\diam330\diam223\diam332\diam001
\]
is a universal partial word for $\{0,1,2,3\}^4$.

Here \begin{itemize}
\item $\left<x_i\right>=\left<0,1,0,1,0,1,0,1,2,3,2,3,2,3,2,3\right>$ \item$\left<y_i\right>=\left<0,1,0,1,2,3,2,3,0,1,0,1,2,3,2,3\right>$ \item$\left<z_i\right>=\left<1,0,3,2,1,0,3,2,1,0,3,2,1,0,3,2\right>$.
\end{itemize}
\end{example}

For $a=2$, the construction yields the cyclic binary universal partial word $001\diam110\diam001$.

\section{Open Problems}\label{open}
While we have constructed an infinite family of universal partial words for $A^4$, increasing $n$ to $5$ already makes brute-force searches intractable, even when accounting for diamondicity.
\begin{question}Does there exist a nontrivial universal partial word over a non-binary alphabet for $n \ge 5$? If so, is it possible to construct a family of such universal partial words?\end{question}

In addition, there are not yet any known nontrivial universal partial words with diamondicity greater than $1$. While such words would be shorter, there are many more initial window frames to check. 
\begin{question}Is there a universal partial word with diamondicity $d>1$?\end{question}

In Section $4$, we were able to find an upper bound on diamondicity for a given $n$, but we would like to find a bound that is a constant fraction of $n$.

\begin{question}
Is there $\varepsilon\in (0,1)$ such that for all $n$ sufficiently large,  every universal partial word for $A^n$ has diamondicity $d \leq \varepsilon n$?
\end{question}

Every non-binary universal partial word is cyclic and has well-defined diamondicity. In contrast, binary universal partial words that are not cyclic and do not have well-defined diamondicity are common. It is known that cyclicity implies pseudocyclicity, which implies well-defined diamondicity, over any alphabet.

\begin{question}
Are well-defined diamondicity, pseudocyclicity, and cyclicity equivalent over a binary alphabet? 
\end{question}
Enumerative questions remain largely unstudied.

\begin{question} For a given $n$ and $A$, how many universal partial words for $A^n$ exist?\end{question}

Chen et al.~\cite{CKSpub} proved the existence of universal partial words over binary alphabets in several cases via Hamiltonian and Eulerian cycles in de Bruijn graphs. The properties of de Bruijn graphs in higher dimension are less studied, so the proof techniques are not readily applicable to larger alphabet sizes.
Other questions that have been studied in the context of de Bruijn sequences and other universal cycles may also be asked.

\begin{question}Given a word $v$ in $A^n$ and a universal partial word $w$ for $A^n$, how can one efficiently search for $v$ in $w$?\end{question}

\section{Acknowledgements}

We thank Jeremy Martin for his feedback and advice throughout the research process. We thank Nathan Graber for his collaboration during the Graduate Research Workshop in Combinatorics (GRWC 2016). We would also like to thank the organizers of GRWC 2016, as well as the other participants who provided insight on the problem. Lastly we thank our reviewers for their helpful suggestions toward publication.

\bibliographystyle{plain}
\bibliography{myreferences}

\bigskip
\noindent
{\em B.\ Goeckner,
Department of Mathematics,
University of Kansas,
Lawrence, KS 66045-7594,}
{\tt bennet@ku.edu} \\

\noindent
{\em C.\ Groothuis,
Department of Mathematics,
University of Nebraska,
Lincoln, NE 68588-0130},
{\tt corbin.groothuis@huskers.unl.edu} \\

\noindent
{\em C.\ Hettle, 
School of Mathematics, 
Georgia Institute of Technology, 
Atlanta, GA 30332-0160,} \\
{\tt chettle@gatech.edu} \\

\noindent
{\em B.\ Kell,
Google,
6425 Penn Avenue,
Pittsburgh, PA 15206,}
{\tt bkell@alumni.cmu.edu} \\

\noindent
{\em P.\ Kirkpatrick, 
Department of Mathematics, 
Lehigh University, 
Bethlehem, PA 18015,}
{\tt prk213@lehigh.edu} \\

\noindent
{\em R.\ Kirsch,
Department of Mathematics,
University of Nebraska,
Lincoln, NE 68588-0130,}
{\tt rkirsch@huskers.unl.edu} \\

\noindent
{\em R.\ Solava, 
Department of Mathematics,
Vanderbilt University,
Nashville, TN 37240,}\\
{\tt ryan.w.solava@vanderbilt.edu} \\

\end{document}